\documentclass[a4paper,11pt, reqno]{amsart}
\usepackage{amsmath,amsfonts,amssymb}

\usepackage{amsthm}

\usepackage{subcaption}
\usepackage{float}
\usepackage{multirow}
\usepackage{adjustbox}
\usepackage{natbib}
\setcitestyle{authoryear,open={(},close={)}}
\usepackage{hyperref}

\newtheorem{lemma}{Lemma}[section]

\usepackage{geometry}
\newcommand{\dsum}{\displaystyle\sum}

\allowdisplaybreaks

\usepackage{setspace}
\allowdisplaybreaks

\let\origmaketitle\maketitle
\def\maketitle{
  \begingroup
  \def\uppercasenonmath##1{} 
  \let\MakeUppercase\relax 
  \origmaketitle
  \endgroup
}
\title[]{\Large The Waste-to-Biomethane Logistic Problem:\\ A mathematical optimization approach}

\author[V. Blanco, Y. Hinojosa \MakeLowercase{and} V. Zavala]{{\large V\'ictor Blanco$^\dagger$, Yolanda Hinojosa$^\ddagger$ and  Victor Zavala$^\star$}\medskip\\
$^\dagger$Institute of Mathematics, Universidad de Granada\\
$^\ddagger$Institute of Mathematics, Universidad de Sevilla\\
$^\star$ Wisconsin Institute for Discovery, University of Wisconsin-Madison}

\date{\today}

\begin{document}

\maketitle

\begin{abstract}
In this paper, we propose a new mathematical optimization approach to make decisions on the optimal design of a logistic system to produce biogas from waste. We provide an integrated model that allows decision makers to optimally determine the locations of different types of plants and pipelines involved in the logistic process, as well as the most efficient distribution of the products (from waste to biomethane) along the supply chain. We analyze the mathematical model and  reduce its size, being able to solve realistic instances in reasonable CPU times. The results of our computational experiments, both in synthetic and in a case study instance, prove the validity of our proposal in practical applications.
\keywords{Logistics; Green Energy; Facility Location; Mathematical Optimization; Biogas; Supply Chain}
\end{abstract}

\section{Introduction}

Decarbonization of energy has emerged as a critical global priority in recent years due to the urgent need to combat climate change. Most World Organizations have recognized the significance of transitioning to cleaner and more sustainable energy sources in the next few years to reduce greenhouse gas emissions. Different agreements have been signed to promote this change. Specifically, the United Nations, through its Sustainable Development Goals, has set different targets to ensure affordable, reliable, sustainable, and modern energy for all by 2030~\citep{ONU}. The International Energy Agency (IEA) has been actively promoting decarbonization efforts by providing policy recommendations, conducting research, and facilitating international cooperation~\citep{IEA}. Additionally, the Intergovernmental Panel on Climate Change (IPCC) \citep{IPCC} has been instrumental in assessing the impacts of climate change and highlighting the importance of decarbonizing the energy sector. In 2019, the European Commission presented the European Green Deal \citep{EGD}, an initiative that aimed to be a statement of intent on the road-map that Europe should follow for the implementation of the United Nations' Sustainable Development Agendas for 2030 and 2050, designed to mitigate the effects of climate change. Various measures have been taken by these organizations, including advocating for renewable energy investments, promoting energy efficiency, encouraging the use of electric vehicles, and supporting the development of innovative technologies such as carbon capture and storage. These collective efforts by world organizations are crucial in driving the global transition towards a low-carbon future and mitigating the adverse effects of climate change.

 One of challenges is to decarbonize the energy system by developing a new power sector based on renewable sources.  It is a known fact that one of the main strategies to achieve the proposed goal is the use of biogas as an alternative renewable energy source to carbon-based energies, since it contributes the reduction of greenhouse gases but also to the development of the circular economy through the anaerobic digestion of organic waste from different sources and its transformation into fuel. Since biomethane is the same molecule as natural gas, it can be distributed via the existing gas distribution networks, facilitating the transition from natural gas to biogas energy. Thus, the installation of anaerobic digestion plants for the conversion of organic waste and livestock manure into biomethane has gained prominence in recent years as a sustainable waste-to-energy solution and many countries are making special interest in the assessment of the installation and adaptation of these  biogas plants.
 For instance, the U.S. Farm Bill \citep[]{USA}, periodically reauthorized by Congress, includes provisions that support the development and utilization of biogas from agricultural sources. The bill provides funding for biogas projects, such as the installation of anaerobic digesters plants on farms or near them, which capture methane emissions from manure and convert them into usable biogas for electricity generation or transportation fuel.

Analyzing the logistic systems behind the implementation of different modes of energy, particularly biogas, is of paramount importance in ensuring the successful and sustainable integration of renewable energy sources into our global energy landscape. Biogas, derived from organic waste materials through anaerobic digestion, presents a promising alternative to conventional fossil fuels. However, its widespread adoption hinges on addressing intricate logistic challenges. Understanding the logistics involved in the collection, transportation, and processing of organic feedstocks is crucial for optimizing efficiency and minimizing environmental impact. Rigorous analyses on the logistical aspects of biogas production in different countries reveal opportunities for streamlining supply chains, enhancing resource utilization, and reducing overall costs~\citep[see, e.g.][]{muradin2018logistic,kwasny2017production}. Effective logistic planning also ensures reliable and consistent feedstock availability, a key factor in the stable operation of biogas plants. By delving into the logistics of biogas implementation, we can develop strategies that not only bolster the economic viability of this renewable energy source but also contribute to the broader goal of achieving a more sustainable and resilient energy infrastructure.

Among the decisions to make in the design of an efficient logistic systems, one of the most critical ones is  selecting the optimal sites to locate the different types of plants involved in the system (biogas plants, pre-treatment plants, and liquefaction plants). This is a difficult task since it involves different agents, different production and conversion technologies, different types of demand centers, and it has to be coordinated with the different distribution processes~\citep{egieya_jafaru_m_biogas_2018}. At this point, mathematical optimization plays a very important role since it is the main tool to determine solutions to complex logistics system like the one we are dealing with here. This fact has caused the publication of a large amount of literature related to this topic. For instance,
 Mathematical optimization has  been used in \citet{TAMPIO2022105-2} for the design of a cost-optimal processing route for a biogas anaerobic digestion plant  to produce fertilizer products based on specified regional needs.
In \citet{BALAMAN2014289} a mixed integer linear programming model is developed to determine the  appropriate locations for the biogas plants and biomass storages.
 \citet{SCARLAT2018915} provides a spatial analysis algorithm that uses data of  manure production and collection  to assess  the spatial distribution of the biogas potential in Europe, in order to decide the location of bioenergy plants.  A geographic information system (GIS) based analysis is used in \citet{VALENTI201850} to determine the size and location of four biogas plants for the Catania province in Sicily. A multi-objective optimization approach for the optimal location of  biogas and biofertilizer plants  at the maximum proﬁt and the minimum environmental impact is presented in \citet{diaz-trujillo_optimization_2019} and it is applied for a geographical region in Mexico.
Mathematical methods have been used to study the location of bioenergy plants in many other regions and countries (\citet{Park2019} in North Dakota, \citet{SULTANA2012192} in the province of Alberta, \citet{egieya_optimization_2020} in Slovenia, \citet{amigun_capacity-cost_2010} in Africa, \citet{silva_multiobjective_2017} in Portugal, \citet{soha_importance_2021} in Hungarian, \citet{iglinski_agricultural_2012} in Poland, \citet{delzeit_impact_2013} in Germany). Furthermore, mathematical optimization tools have been already recognized as a crucial tool to make decisions in the design of systems for biogas production. \cite{hernandez2017optimal},  propose a non linear optimization model to compute the optimal operating conditions of the reactors, the biogas composition and the content of nutrients in the digestate in the production of biodiesel from waste via anaerobic digestion. \citet{hu2018supply}  give a mathematical optimization framework to analyze the economical and environmental benefits of recovering biogas, caproic acid (C6), and caprylic acid (C8) from different sources of organic waste. \citet{ankathi2021gis} consider a mixed integer linear optimization model to determine the location and capacities of biogas plants based only in the location and production of the farms. In terms of locational decisions, some of the classical mathematical approaches have been already adapted to incorporate \textit{green goals}~\citep[see e.g.][]{martinez2017green}, although more advances are expected within the next few year based on what our society needs.

On the other hand,  mathematical optimization is recognized as a fundamental tool for the design and modeling of multiple logistics, transportation, and supply chain  problems \citep[see e.g.][among many others]{hinojosa2000multiperiod, hinojosa2008dynamic,blanco2010planning,blanco2020optimization,blanco2022network, blanco2023pipelines,blanco2024optimal}. Specifically, is particularly useful to construct robust and effective supply chains to integrate  bioenergy in any economy 
  \citep[see e.g.][and the references therein]{DEMEYER2015247,   sarker_optimal_2018, TOMINAC2022107666,   CAMBERO201462, GHADERI2016972}  . In particular, in order to implement an efficient and sustainable biogas distribution system, it is necessary to properly design a robust logistic plan for all the elements involved in the process, such as manure, biomethane, waste, etc.   In this phase,
  it is crucial to decide not only where to locate the anaerobic digestion plants but also,  where to locate pre-treatment plants for cleaning and drying the waste or transshipment plants to distribute the products, how to collect the manures from the farms or fields and sending them to the plants, how to link (if needed) the different types of plants, how to distribute the final biomethane to the gas distribution network or to external clients, how to dispatch possible fertilizers back to some of the farms, etc. 
A few works have already proposed models to make decision in this setting.  \citet{jensen_optimizing_2017} propose minimum cost flow based model for finding the optimal production and investment plan for a biogas supply chain by means of minimizing the transportation cost on an existing supply chain network. 
Three layers of analysis for designing optimal animal waste supply for anaerobic bio-digestion are detailed in  \citet{MAYERLE201646}: (1) Identification of the optimal locations for the anaerobic digestion plants the farms' information; (2) specification of the optimal distribution system; and (3) an  operational layer that includes scheduling optimal biomass collection from each farm to minimize biogas loss.
\citet{sarker_modeling_2019} studies the optimization of the supply chain cost for a biomethane gas production system which is organized in four stages (collecting feedstock to hubs located according to zip code areas, transporting feedstock from hubs to reactor(s), transporting biomethane gas from reactor(s) to condenser(s), and  shipping the liqueﬁed biomethane gas from condensers to final demand points). The problem is formulated as a  mixed-integer non-linear mathematical optimization problem and, due to the complexity of the non-lineal model, a genetic heuristic algorithm is propose to solve it.
 

  In this paper, we introduce the  Waste-to-Biomethane Logistic Problem and provide a general and flexible mathematical optimization model 
 to make decisions of locating biogas, pre-treatment, and liquefaction plants and distributing the different types of elements along a complex network that integrates all the aspects mentioned above. Thus, we consider a supply chain for the biomethane gas production combining six stages: (A) collecting waste and sending it to pre-treatment plants (independently of the zip code); (B) delivering contaminant-free organic waste from pre-treatment plants to anaeorobic digestion plants; (C) constructing pipelines for transporting  biomethane from anaerobic digestion plants to injection points of an existing gas pipeline network;  (D) constructing pipelines for transporting  biomethane from anaerobic digestion plants to biomethane liquefaction plants; (E) dispatching fertilizer from anaerobic digestion plants to waste sources and; (F) shipping liquefied biomethane from biomethane liquefaction plants to customer demand points. 
We jointly integrate all these complex stages into a mixed integer linear programming (MILP) model seeking to minimize the overall transportation cost of the system by assumming that a limited budget is given to install all the plants and pipelines. Our MILP model can be efficiently solved using  off-the-shelf optimization solvers (as Gurobi, CPLEX, or FICO) after proving some theoretical results for reducing the number of variables. This model is applied and tested in a real-word dataset based on the region of upper Yahara Watershed in the state of Wisconsin   \citep[see][for further information about this dataset]{ SAMPAT2019352, SAMPAT2021105199}.

The remainder of this paper is organized as follows. In Section \ref{sec:problem}  we introduce the Waste-to-Biomethane Logistic Problem (W2BLP) and present our modeling assumptions. Section \ref{sec:model} is devoted to detail the mathematical optimization model that we develop for the problem. We also prove in this section a theoretical result that allows us to significantly reduce the number of variables in the model.  In Section \ref{sec:comput}, we report the results of our computational experiments on realistic synthetic instances. A case study is presented in Section \ref{sec:case-study} where the model is applied to the real-world dataset based on the region of upper Yahara Watershed.
The paper closes in Section \ref{sec:conclu} with some conclusions and future research.

\section{The Waste-to-Biomethane Logistic Problem \label{sec:problem}}

The problem under analysis consists on efficiently use the agricultural waste, as livestock manure, energy crops, municipal waste, etc,  to be transformed into biomethane.  This transformation requires different phases.
First, from a given  set of waste sources (WS), as farms or residual storages, each of them producing an amount of waste, it is transported into pre-treatment (PT) plants where the non-organic material is removed, and what remains is dried, pressed and prepared. The location of the PT plants is to be determined among a finite set of potential positions.

The dried contaminant-free organic waste (DOW) obtained after the pre-treatment, is delivered to an anaerobic digestion (AD) plant, where it is transformed into biomethane. 
The position of the AD plants is also to be decided among a finite set of potential locations. Biomethane is then either directly injected into an existing gas pipeline network (GPN) or processed into liquified natural gas (LNG). We assume that we are given a finite set of injection points in the GPN and that pipelines connecting the AD plants with a selected set of injections point are to be built. A minimum percent of the produced biomethane is to be injected in the GPN, and the remainder biomethane is transformed into LNG and distributed to a given set of extra customers (EC), each of them with a given demand. This minimum percent represents the trade-off between the gas injection to the network with respect to the amount of gas served to the external customers. The liquefaction of biomethane requires the construction of  Biomethane Liquefaction (BL) plants as well as pipelines connecting the AD plants with them. Then, the liquified gas is delivered to the customers using tank trucks. Finally,  the remaining material in the DOW to biomethane transformation process (digester solids) is then returned back from the AD plants to some of the WS where it can be used as fertilizer. In Figure \ref{fig:prob} we illustrate the different elements that appear in this logistic problem. The names are differentiated by color. In blue color we indicate the input data (WS, GPN and EC), whereas in green color we highlight the decisions to be made on: the (PT, AD and BL) plants and the (to GPN and BL) pipelines to be installed. The transportation routes are highlighted with red arrows in the plot.

\begin{figure}[h]
\begin{center}
\includegraphics[width=0.95\textwidth]{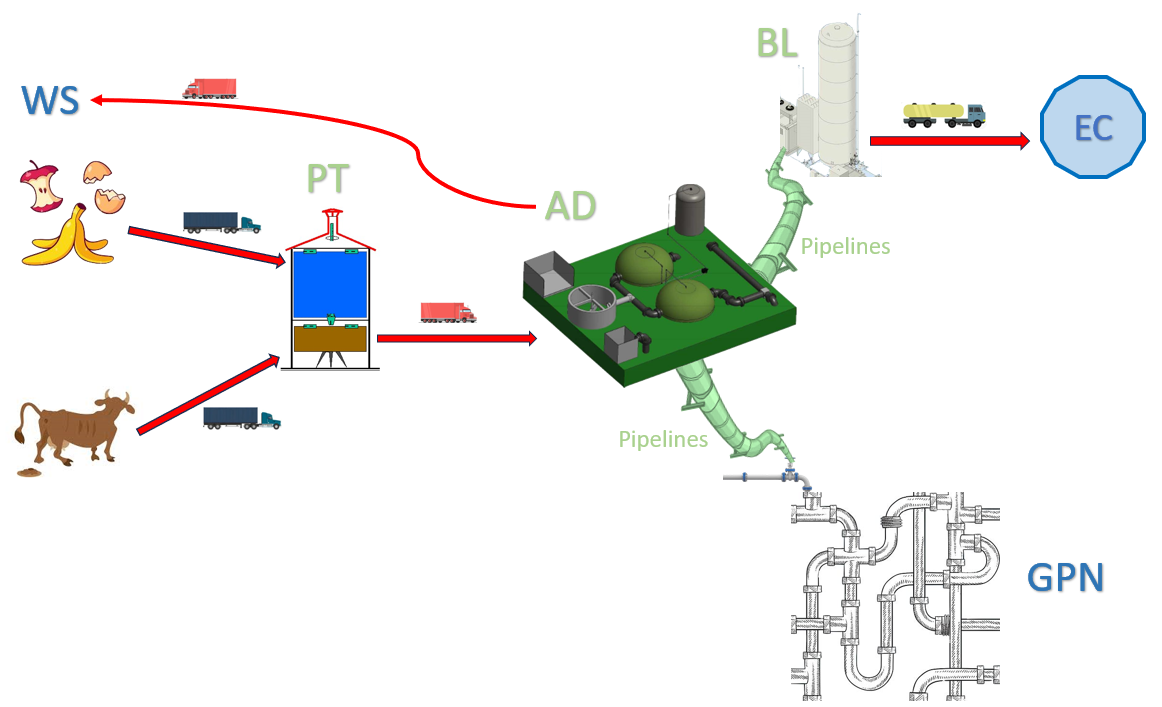}
\caption{The process modeled in the W2BLP.\label{fig:prob}}
\end{center}
\end{figure}

The following hyphoteses are assumed in our model, based on the technological requirements of the process:
\begin{itemize}
\item A minimum percentage of the total amount of waste in the WS must be collected and sent to the PT plants to satisfy the gas requirements of the GPN. Since a given percent of the produced biomethane is to be injected in the GPN, this hypothesis implies that a minimum amount of biogas must be produced to be injected to the network.
\item Each PT plant sends its whole produced DOW to the AD plants. In the AD plants, the DOW received is transformed into biomethane. The digester solids of this transformation can be delivered to the WS as fertilizers.
\item We assume that each WS receives a given proportion of the total amount of digester solid produced at the AD plants. In case this proportion is $0$, the WS does not receive any digester solid (as in the case of residual storages that may not be interested in the product). 
We do not force delivering the whole digester solid produced in the AD plants to the WS.
\item The whole biomethane production obtained in the AD plants must be sent.  It is shared into the GPN and the BL plants but, as already mentioned, a minimum percentage of the total biomethane production is required to be injected  into the GPN.
\item Each BL plant sends its whole LNG production to the EC.
\item The demand of the EC is not assumed to be fully satisfied. Instead, we consider that every EC receives a pre-specified proportion of 
the total LNG produced at the BL plants. This proportion may depends on its demand or other preference criterion that the agents consider reasonable.

\item A percent of the amount (weight and volume) of the different products is assumed to be lost at each phase of the whole process. Specifically, a given percentage of the amount of waste, DOW, and biomethane received in PT, AD and BL, respectively, is lost during the process. This technical hyphotesis is derived from the chemical processes to obtain the different types of products from their raw materials. 
\end{itemize}

The efficient management of this system requires both deciding the location of the different types of plants (PT, AD, and BL) and pipelines (connecting AD with GPN and AD with BL) to be constructed, as well as the distribution of the product along the different phases of this process.

There is a cost for opening each potential location of each different type of plants, and a construction cost for unit length for the pipelines. We assume that a budget is given to install all the plants and the pipelines.

For the distribution of the products we assume that a unit transportation cost is provided for the different phases, namely, delivering from WS to the PT plants (waste), from PT to the AD plants (DOW), from the BL plants to the EC (LNG), and from AD to the WS (digester solid).

The Waste-to-Biomethane Logistic Problem (W2BLP) consists of deciding the location of the different plants and pipelines, and the way the product is distributed, minimizing the overall transportation cost with the given installation budget.

\section{Mathematical Optimization Model\label{sec:model}}

In this section we present the mathematical optimization model that we propose for the W2BLP. First, we define the parameters and variables that we use in our model.

\subsection{Parameters}

The following parameters are assumed to be given to our model.

\subsubsection{Index Sets}

\noindent$-$ $N =\{1, \ldots, n\}$: Index set for the set WS.\\
$-$ $P_1=\{1, \ldots, m_1\}$: Index set for the set of potential PT plants.\\
$-$  $P_2=\{1, \ldots, m_2\}$: Index set for the set of potential AD plants.\\
$-$ $P_3=\{1, \ldots, m_3\}$: Index set for the set of potential BL plants.\\
$-$  $I=\{1, \ldots, |I|\}$: Index set for Injection Points in GPN.\\
$-$ $C=\{1, \ldots, |C|\}$: Index set for the set EC.

\subsubsection{Production parameters}

\noindent$-$ $w_i$: waste produced at WS $i\in N$.  \\
$-$ $\mathrm{W}= \dsum_{i\in N} w_i$: Overall production of waste in all the WS. \\
$-$ $\delta$: Proportion of waste transformed into DOW at the PT plants ($0<\delta<1$).\\
$-$ $\gamma_1$: Proportion of DOW transformed into biomethane at the AD plants ($0<\gamma_1<1$).\\
$-$ $\gamma_2$: Proportion of DOW transformed into fertilizers at the AD plants ($0<\gamma_2<1$). We assume that $\gamma_1+\gamma_2<1$.\\
$-$ $\beta$: Proportion of biomethane transformed into LNG at the BL plants ($0<\beta<1$).\\
$-$ $D_\ell$: Demand of EC $\ell\in C$.\\
$-$ $p$:  Lower bound percent of the total waste produced at the WS that must be collected and sent to the PT plants ($0<p\leq 1$).\\
$-$ $q$: Lower bound percent of the total biomethane produced at the AD plants that must be injected  into the GPN ($0<q\leq 1$).\\
$-$ $R_i$: Proportion of the total amount of fertilizer to be deliver to WS $i\in N$ ($\dsum_{i\in N} R_i\leq 1$).\\
$-$ $\alpha_\ell$: 
 Proportion of the total amount of  LNG at the BL plants that will be received by demand point $\ell\in C$ 
($\dsum_{\ell\in C} \alpha_\ell=1$).

\subsubsection{Costs}

{\bf Set-up Costs:}

These costs represent the installation  costs of the different type of plants and links.
They may include not only the construction cost of the different plants and pipelines, but also other costs related with the maintenance and use of them, as labor costs, land costs, etc.

\noindent$-$  $f^1_j$: Set-up cost for opening PT plant $j\in P_1$.\\
$-$  $f^2_j$: Set-up cost for opening AD plant $j\in P_2$.\\
$-$ $f^3_j$: Set-up cost for  opening BL plant $j\in P_3$.\\
$-$ $h_{jk}$: Set-up cost for installing a pipeline linking  AD plant $j\in P_2$ with BL plant $k\in P_3$ .\\
$-$ $g_{j\ell}$: Set-up cost for installing a pipeline linking  AD plant $j\in P_2$ with injection point in GPN $\ell \in I$.\\
$-$ $\mathrm{B}$: Total budget for installing plants and pipelines.\\

\noindent {\bf Transportation Costs:}

These costs represent the transportation costs of the different products (waste, DOW, LNG, and digester solid) 
that are usually transported  by trucks from the different geographical locations of the elements in the system.

\noindent $-$ $c^1_{ij}\geq 0$: Unit transportation cost of  waste from WS $i\in N$ to PT plant $j\in P_1$.\\
$-$ $c^2_{jk}\geq 0$: Unit transportation cost of DOW from PT plant $j\in P_1$ to AD plant $k\in P_2$.\\
$-$ $c^C_{j\ell} \leq 0$: Unit transportation cost (profit) of LNG  from BL plant $j\in P_3$ to EC $\ell \in C$.  Note that this non positive cost represent the benefit of delivering to  the EC each unit of liquified gas for their particular use.\\
$-$ $c^N_{ji} \geq 0$: Unit transportation cost of digester solid  from AD plant $j\in P_2$ to WS $i\in N$.

Note that the transportation costs for the biogas through the pipelines are already included in the set-up costs described above.

\subsection{Variables}

Our model makes decision on the different plants and pipelines that are opened, as well as the amount of product delivered at the different phases of the process.

\noindent{\bf Binary variables:}

The following variables determine whether a plant or a pipeline should be installed.

$
y_{j}^1 = \begin{cases}
1 & \mbox{if PT $j$ is open,}\\
0 & \mbox{otherwise}
\end{cases} \ \ \forall j\in P_1, \ \
\quad y_{j}^2 = \begin{cases}
1 & \mbox{if AD $j$ is open,}\\
0 & \mbox{otherwise}
\end{cases}\ \ \forall j\in P_2,
$

$
y_{j}^3 = \begin{cases}
1 & \mbox{if BL $j$ is open,}\\
0 & \mbox{otherwise}
\end{cases} \ \ \forall j\in P_3, \ \
\quad z_{jk} = \begin{cases}
1 & \mbox{if  a pipeline linking AD $j$ }\\
& \mbox{ with BL $k$ is built $\qquad\forall j\in P_2, k\in P_3$,}\\
0 & \mbox{otherwise}
\end{cases}  \ \ 
$

$$
t_{j\ell} = \begin{cases}
1 & \mbox{if a pipeline linking AD plant $j$ with }\\
& \mbox{injection point in GPN $\ell$ is built }\\
0 & \mbox{otherwise}
\end{cases}  \qquad \forall\, j\in P_2,\quad \ell\in I,
$$

\noindent{\bf Continuous Variables:}

The following variables of our model decide the amount of product delivered between the different sites in the system.

$
x_{ij}^1:$ amount of waste delivered from WS $i$ to the PT plant $j$, $\forall$ $i\in N$, $j\in P_1$.

$
x_{jk}^2:$ amount of DOW delivered from PT plant $j$ to AD plant $k$, $\forall$ $j\in P_1$, $k\in P_2$.

$
x_{jk}^3: $ amount of biomethane delivered from AD plant $j$ to BL plant $k$, $\forall$ $j\in P_2$, $k\in P_3$.

$
x_{j\ell}^I: $ amount of biomethane delivered from AD plant $j$ to GPN's injection point $\ell$, $\forall$ $j\in P_2$, $\ell\in I$.

$
x_{ji}^N:$  amount of digester solid delivered from AD $j$ to WS $i$, $\forall$ $j\in P_2$, $i\in N$.

$
x_{j\ell}^C: $ amount of LGN delivered from BL plant $j$ to EC $\ell$, $\forall$ $j\in P_3$, $\ell\in C$.

\subsubsection{Objective Function}

The goal of the W2BLP is to minimize the total transportation cost of the system. For that, one must decide the optimal position of the different plants and pipelines with the given budget $B$. These locations have a direct impact in the transportation cost.

With the above set of variables, the overall transportation cost can be written as follows:
\begin{align}
\dsum_{i\in N}\dsum_{j\in P_1} c_{ij}^1 x^1_{ij} + \dsum_{j\in P_1}\dsum_{k\in P_2} c_{jk}^2 x^2_{jk}
     + \dsum_{j\in P_2}\dsum_{i \in N} c^N_{ji} x^N_{ji} 
      +\dsum_{j \in P_3}\dsum_{\ell\in C} c^C_{j\ell} x^C_{j\ell}
      \label{of:transport}.
\end{align}

\subsection{Constraints}

The assumptions of our problem  
are adequately 
established by the following constraints:

\begin{itemize}
\item The budget for installing the different plants and pipelines must be satisfied:
\begin{equation}
\dsum_{j\in P_1} f^1_j y_j^1 + \dsum_{j\in P_2} f_j^2 y_j^2+ \dsum_{j\in P_3} f^3_j y_j^3 + \dsum_{j\in P_2}\dsum_{k\in P_3}h_{jk}z_{jk}+\dsum_{j\in P_2}\dsum_{\ell\in I}g_{j\ell}t_{j\ell}\leq \mathrm{B}\label{con:setup}
\end{equation}
\item \textit{Flow Conservation Constraints}: The product must be adequately routed through the intermediate stages of the supply chain and enforced by the following constraints:
  \begin{equation}
\dsum_{k\in P_2} x^2_{jk} = \delta \dsum_{i\in N} x^1_{ij} , \forall j \in
P_1,\label{flow:2}
\end{equation}
\begin{equation}
\dsum_{k\in P_3} x^3_{jk}+\dsum_{\ell\in I} x^I_{j\ell} = \gamma_1 \dsum_{k\in P_1} x^2_{kj} , \forall j \in
P_2.\label{flow:3}
\end{equation}
 \begin{equation}
\dsum_{i\in N} x^N_{ji} = \gamma_2 \dsum_{k\in P_1} x^2_{kj}, \forall j \in
P_2,\label{flow:5}
\end{equation}
\begin{equation}
\dsum_{\ell\in C} x^C_{j\ell} = \beta \dsum_{k\in P_2} x^3_{kj}, \forall j \in
P_3,\label{flow:4}
\end{equation}

Constraints \eqref{flow:2} ensure that all the waste received at a PT plant is delivered to the AD plants once it is processed ($100\delta\%$ of the waste is transformed into DOW). $100\gamma_1\%$ of the amount of DOW received at each AD plant is transformed into biomethane and fully delivered either to the injection points or to the BL plants (Constraint \eqref{flow:3}) and $100\gamma_2\%$ is converted into digester solid and delivered to the WS (Constraint \eqref{flow:5}). Finally, Constraints \eqref{flow:4} assure that the amount of biomethane received at the BL plants is converted into $100\beta\%$ of LNG and fully ddelivered to the EC.

\item \textit{Demand and Capacity Constraints:} The different products obtained in the process are adequately delivered according to  demands and capacities. Specifically, the WS must send waste and receive  digester solid, the EC must receive  LNG and the GPN biomethane.
To this end, the following constraints are incorporated to our model:
\begin{equation}
  \dsum_{j\in P_1} x^1_{ij} \leq w_i, \forall i \in N,\label{flow:1}
  \end{equation}
   \begin{equation}\label{flow:q2}
 \dsum_{i\in N}\dsum_{j\in P_1}  x^1_{ij}\geq pW
  \end{equation}
   \begin{equation}\label{flow:q}
 \dsum_{j\in P_2}\dsum_{\ell\in I}  x^I_{j\ell}\geq q\gamma_1 \dsum_{k\in P_1} \dsum_{j\in P_2} x^2_{kj} 
 \end{equation}
\begin{equation}
\dsum_{j\in P_3} x^C_{j\ell}\leq D_{\ell}, \forall \ell \in
C.\label{con:demand}
\end{equation}

Constraints \eqref{flow:1} assure that the amount delivered from each WS is at most the waste it produces. Constraint \eqref{flow:q2} forces to send, from the WS's, at least $100p\%$ of the total amount of waste.   Constraint \eqref{flow:q} states that at least $100q\%$ of the produced biomethane  must be injected in the GPN. Each EC, by Constraints \eqref{con:demand}, at most receive its required demand. Note that the transportation cost for the product in this phase is non positive, and then, if possible, this demand will be satisfied. 

Note also that, by \eqref{flow:2}, Constraints \eqref{flow:q2} and \eqref{flow:q} can be equivalently rewritten as:
 \begin{equation}\label{flow:qq}
  \dsum_{j\in P_2}\dsum_{\ell\in I}  x^I_{j\ell}\geq pq\delta\gamma_1 W
  \end{equation}

It may happen that if $q$ is large, the overall demand for the EC that can be satisfied is small. Thus, one can establish a sharing rule for the available LNG, i.e., a vector $(\alpha_1, \ldots, \alpha_{|C|}) \in [0,1]^{|C|}$ representing the percent of the LNG at the BL plants that will be received by each EC. The following constraints assure the adequate verification of this share:
\begin{equation}
\dsum_{k\in P_3} x^C_{k\ell} = \alpha_\ell\beta \dsum_{j\in P_2}\dsum_{k\in P_3} x^3_{jk} , \forall \ell \in
C.\label{flow:7}
\end{equation}

Finally, Constraints \eqref{flow:6} force that each WS receives its given percentage of the total production of digester solid:
\begin{equation}
\dsum_{j\in P_2} x^N_{ji} =  R_i\gamma_2 \dsum_{j\in P_1}\dsum_{k\in P_2} x^2_{jk}, \forall i \in
N,\label{flow:6}
\end{equation}

  \item \textit{Distribution through open Plants and Links}: The transportation of the product in the different phases forces the plants and/or pipelines to be installed. This is assured by the following constraints:
  \begin{align}
 \dsum_{i\in N} x^1_{ij} &\leq pW y_j^1, &\forall j \in
P_1,\label{con:x1-y1-2}\\
 \dsum_{k\in P_1} x^2_{kj} &\leq p\delta W  y_j^2 , &\forall j \in
P_2,\label{con:x2-y2-2}\\
 x^3_{jk} &\leq p\delta \gamma_1 W  z_{jk} , & \forall j \in P_2, \forall k \in
P_3,\label{con:x3-z2}\\
  x^I_{j\ell} &\leq p\delta \gamma_1 W  t_{j\ell} , &\forall j \in P_2, \forall \ell \in
I.\label{con:x5-z2}
\end{align}
These constraints assure that in case a plant or a link is not open, the product cannot be distributed through the corresponding plant or link.
\item \textit{Compatibility of open plants and links}: The installation of a pipeline linking a plant with other plant or with an injection point in the GPN, is subject to the previous instalation of the plants that the pipelines are connecting. These conditions are imposed in our model with the following linear constraints:
\begin{align}
&z_{jk}\leq y_k^3, \forall j\in P_2, k\in P_3\label{con:z-y3}\\
&z_{jk}\leq y_j^2, \forall j\in P_2, k\in P_3\label{con:z-y2}\\
&t_{j\ell}\leq y_j^2, \forall j\in P_2, \ell\in I\label{con:zp-y2}
\end{align}
\end{itemize}

Summarizing all the previous descriptions, the W2BLP can be model with the following mathematical optimization problem that we denote by $(W2BLP)_0$:

\begin{align*}
\min
     &  \dsum_{i\in N}\dsum_{j\in P_1} c_{ij}^1 x^1_{ij} + \dsum_{j\in P_1}\dsum_{k\in P_2} c_{jk}^2 x^2_{jk}
     + \dsum_{j\in P_2}\dsum_{i \in N} c^N_{ji} x^N_{ji} 
      +\dsum_{j \in P_3}\dsum_{\ell\in C} c^C_{j\ell} x^C_{j\ell}\\
      \mbox{s.t. } & \eqref{con:setup}-\eqref{flow:1}, \eqref{con:demand}-\eqref{con:zp-y2},\\
&y^1_j \in \{0,1\}, \forall j \in P_1,\\
&y^2_j \in \{0,1\}, \forall j\in P_2,\\
&y_{j}^3 \in \{0,1\}, \forall j\in P_3,\\
&z_{jk}\in \{0,1\}, \forall j\in P_2, k\in P_3,\\
&t_{j\ell}\in \{0,1\}, \forall j\in P_2, \ell\in I,\\
&x_{ij}^1 \geq 0, \forall i\in N, j\in P_1,\\
&x_{jk}^2 \geq 0, \forall j\in P_1, k\in P_2,\\
&x_{jk}^3 \geq 0, \forall j\in P_2, k\in P_3,\\
&x_{j\ell}^I \geq 0, \forall j\in P_2, \ell\in I,\\
&x_{ji}^N \geq 0, \forall j\in P_2, i\in N,\\
&x_{j\ell}^C \geq 0, \forall j\in P_3, \ell\in C.
\end{align*}
The above model is a mixed integer linear programming (MILP) problem,  which can be difficult to solve in commercial solvers for real-world instances. In the following result we show that the size of the model can be considerably reduced based on the fact that once decided the potential AD plant which is open, it will optimally distribute the biogas to the GPN through its less costly injection point.

\begin{lemma}
Let $\bar S = (\bar y^1, \bar y^2, \bar y^3, \bar z, \bar t, \bar x^1, \bar x^2, \bar x^3,\bar x^I,\bar x^N,\bar x^C)$ be an optimal solution of the model $(W2BLP)_0$. If $\bar t_{j\ell}=1$ for some  $j\in P$ and $\ell \in I$ then, there exists an optimal solution $\hat S = (\bar y^1, \bar y^2, \bar y^3, \bar z, \hat t, \bar x^1, \bar x^2, \bar x^3,\hat x^I,\bar x^N,\bar x^C)$ with:
$$
\hat t_{j\ell(j)}=1,\quad \hat x^I_{j\ell(j)} = \sum_{\ell\in I} \bar x^I_{j\ell},\qquad \hat t_{j\ell}=0,\quad \hat x^I_{j\ell} =0,\quad  \forall \ell\neq \ell(j),
$$
where $\ell(j)= \arg\min_{\ell\in I} g_{j\ell}$.

Furthemore, the overall set-up costs of $\hat S$ is smaller or equal than the ones for $\bar S$.
\end{lemma}
\begin{proof}
The proof follows straightforward since there are no transportation costs between the AD plants and the injections points but linking them is accounted in the budget constraint. Thus, replacing the link $(j,\ell)$ by $(j,\ell(j))$ does not affect the objective function but reduces the set-up costs.
\end{proof}

The result above allows us to simplify the model by reducing the number of variables modeling the links and the flows between the AD plants and the GPN. Specifically, one can replace the two-index variables $t$ and $x^I$ by one-index variables that, abusing of notation, we denote as:

$$
t_{j} = \begin{cases}
1 & \mbox{if a pipeline linking AD plant $j$ with }\\
& \mbox{injection point $l(j)$ is built }\\
0 & \mbox{otherwise}
\end{cases} \qquad \forall\, j\in P_2,
$$

and

$x_{j}^I: $ amount of biomethane delivered from AD plant $j$ to GPN's injection point $l(j)$,  $\forall\, j\in P_2$.

Thus, if we denote by $g_j = g_{j\ell(j)}$ for all $j\in P_2$,
one can replace the two-indices $t$ and $x^I$variables in model $(W2BLP)_0$ by the one-index $t$ and $x^I$variables above. Consequently, Constraints \eqref{con:setup}, \eqref{flow:3}, \eqref{flow:qq}, and \eqref{con:x5-z2} are replaced by the following inequalities and equations:
\begin{align}
\dsum_{j\in P_1} f^1_j y_j^1 + \dsum_{j\in P_2} f_j^2 y_j^2+ \dsum_{j\in P_3} f^3_j y_j^3 + \dsum_{j\in P_2}\dsum_{k\in P_3}h_{jk}z_{jk}+\dsum_{j\in P_2}g_{j}t_{j}\leq \mathrm{B}, \label{con:setupB}\\
\dsum_{k\in P_3} x^3_{jk}+ x^I_{j} = \gamma_1 \dsum_{k\in P_1} x^2_{kj} , \forall j \in
P_2.\label{flow:3B}\\
 \dsum_{j\in P_2} x^I_{j} \geq pq\delta\gamma_1 W,\label{flow:qB}\\
  x^I_{j} \leq p\delta \gamma_1 W  t_{j} , \forall j \in P_2.\label{con:x5-z2B}
\end{align}

As one can observe in the next section, the reduced model above is able to solve real instances of the W2BLP in reasonable CPU time. Our approach is then a useful tool for making decisions on the logistic of biogas production system.

\section{Computational Experiments \label{sec:comput}}

We have run a series of experiments to analyze the computational performance of our approach. The main goal of the experiments is to determine the computational limitations of our model and its ability to obtain solutions for real-world instances.

We have randomly generated different instances with different sizes and parameters. We generate the coordinates for the WS, the potential location for the different plants (PT, AD,and BL), the injection points in GPN and the EC uniformly in $[0,1000] \times [0,1000]$. For the sake of simplification, we assume that the number and the location of WS ($n$) is the same that the number and the location of potential PT plants ($m_1$), and that the number and the location of potential AD plants ($m_2$) and potential BL plants ($m_3$) are also the same. Aditionally, the number of EC  ($|C|$) also coincides with the number of injection points in GPN ($|I|$). The value of $n=m_1$ ranges in $\{25, 50, 100, 200, 500\}$, $m:=m_2=m_3$
ranges in $\{5, 10, 20, 50, 100\}$ with $m \in [\lceil\frac{n}{10}\rceil, \lfloor\frac{n}{2}\rfloor)$, and $d:=|I|=|C|$ ranges in $\{10, 20, 50, 100\}$ with $d \leq \lceil \frac{n}{2} \rceil$ (here $\lceil \cdot \rceil$ and $\lfloor \cdot \rfloor$ stand for the ceiling and floor integer rounding functions, respectively). 

Let ${\rm dist}_{ij}$ denote the Euclidean distance between locations $i$  and $j$.
We considered as unit transportation costs the Euclidean distances between the different points (${\rm dist}_{ij}$), except for the transportation costs of LNG from BL plants to the EC where the unit transportation cost is considered as a profit and it is defined as:
$$
c_{j\ell}^C = -\frac{\displaystyle\max_{k \in P_3, l \in C} ({\rm dist}_{kl})^2}{1+ {\rm dist}_{j\ell}}
$$
These costs are designed to represent that as closer the extra customer to the open BL plant the larger the profit to satisfy its unit demand.

The set-up costs for the PT plants and the BL plants were chosen all equal to one unit (in million of \$) and the installation cost for the AD plants wer fixed to $5$ units (in million of \$). The set-up costs for the pipelines linking the AD plants with the injection points in GPN and the BL plants are 
0.1 times the normalized (by the maximum) Euclidean distance between the corresponding points.

For determining adequate budgets for a given instance, we first solve the problem without the budget constraint \eqref{con:setup}. After solving the problem, we compute its effective set-up cost by adding up the costs of the open plants and pipelines that route a positive flow. We denote by $\widehat{B}$ this set-up cost and consider as budget $\mathrm{B} = 0.2  \widehat{B}$. This budget indicates that one can only use $20\%$ of the cost of makig the best decision with no budget.

The slurry production at each farm ($w_i$ for $i\in N$) has been uniformly generated in $[1,100] \cap \mathbb{Z}_+$ (here $\mathbb{Z}_+$ stands for the set of nonnegative integer numbers). The percent of the overall produced biomethane that must be injected into the existing GPN, $q$, ranges in $\{50\%, 70\%, 90\%\}$. We have considered parameters $p=1$, $\delta=0.8$, $\gamma_1=0.8$, $\gamma_2=0.15$, $\beta=0.7$, $\alpha_\ell= \frac{D_{\ell}}{\sum_{l\in C} D_l}$ for all $\ell\in C$, and $R_i=\frac{w_i}{W}$ for all $i\in N$. The LNG demand of each EC ($D_\ell$ for $\ell\in C$) was randomly generated in $[\delta \gamma^1 \beta, 100 \delta \gamma^1 \beta]$.

The values of the parameters that we consider in our experiments are summarized in Table \ref{table:params}.

\renewcommand{\arraystretch}{1.3}
\begin{table}[h]
\adjustbox{width=\textwidth}{\begin{tabular}{c|c|c|c}
{\bf Coordinates} & $n = m_1$ & $m:=m_2=m_3$ & $d:=|I|=|C|$\\\hline\hline
Unif$[0,1000]^2$ & $\{25, 50, 100, 200, 500\}$ & $\{5, 10, 20, 50, 100\}$ with $m \in [\lceil\frac{n}{10}\rceil, \lfloor\frac{n}{2}\rfloor)$ &  $\{10, 20, 50, 100\}$ with $d \leq \lceil \frac{n}{2} \rceil$\\\hline
\multicolumn{4}{c}{}\\
$c^1_{ij},\; c^2_{ij},\; c^N_{ij}$ & $c^C_{j\ell}$ & $f^2$ & $f^1$, $f^3$ \\\hline\hline
${\rm dist}_{ij}$ & $-\frac{\displaystyle\max_{k \in P_3, l \in C} ({\rm dist}_{kl})^2}{1+ {\rm dist}_{j\ell}}$ & $5$ & $1$\\\hline
\multicolumn{4}{c}{}\\
$h_{jk}$ & $g_j$ & $B$ & $w_i$ \\\hline\hline
$0.1\times\frac{{\rm dist}_{jk}}{\max_{j', k'} {\rm dist}_{j'k'}}$ & $0.1\times\min_{\ell} \frac{{\rm dist}_{j\ell}}{\max_{k', \ell'} {\rm dist}_{k'\ell'}}$ & $0.2 \widehat{B}$ & Unif$[1,100]\cap \mathbb{Z}_+$\\\hline
\multicolumn{4}{c}{}\\
$q$ & $\delta=\gamma_1$ & $\gamma_2$ & $\beta$\\\hline\hline
$\{50\%, 70\%, 90\%\}$ & $0.8$ & $0.15$ & $0.7$\\\hline
\multicolumn{4}{c}{}\\
$p$ & $\alpha_\ell$ & $R_i$ & \multicolumn{1}{c}{$D_\ell$} \\\cline{1-4}\noalign{\vskip\doublerulesep
         \vskip-\arrayrulewidth}\cline{1-4}
1&$\frac{D_{\ell}}{\sum_{l\in C} D_l}$ & $\frac{w_i}{W}$ & \multicolumn{1}{c}{Unif$[\delta \gamma^1 \beta, 100 \delta \gamma^1 \beta]$} 
\end{tabular}}
\caption{Summary of the parameters used in our experiments.\label{table:params}}
\end{table}
\renewcommand{\arraystretch}{1}
The model has been coded in Python 3.7 in an iMac with 3.3GHz with an Intel Core i7 with 4 cores and 16GB 1867 MHz DDR3 RAM. We used Gurobi 9.1.2 as optimization solver. A time limit of 6 hours was fixed for all the instances. All the instances with $n\leq 100$ were optimally solved. For the larger instances we fixed a MIP Gap limit (the solver stops when reaching such a limit and outputs the best feasible solution). For $n=200$  we fix a MIPGap limit of $2\%$, for $n=500$ of $4\%$. For each combination of parameters $n, m, d, q$ we solved five instances. Thus, we have solved a total of $450$ instances.

\begin{figure}
\begin{subfigure}[b]{0.48\textwidth}
\includegraphics[width=\textwidth]{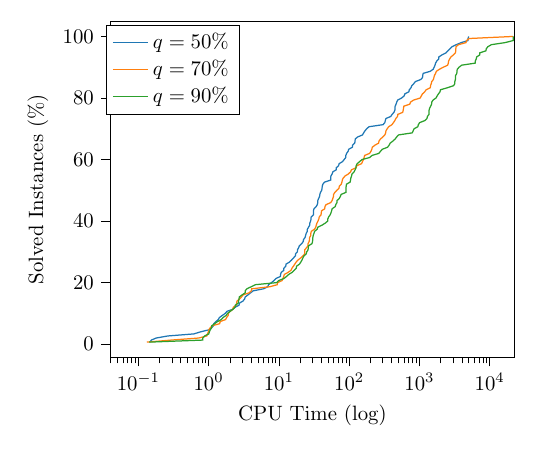}
\caption{CPU times (log scale).}
\end{subfigure}\quad
     \begin{subfigure}[b]{0.46\textwidth}
\includegraphics[width=\textwidth]{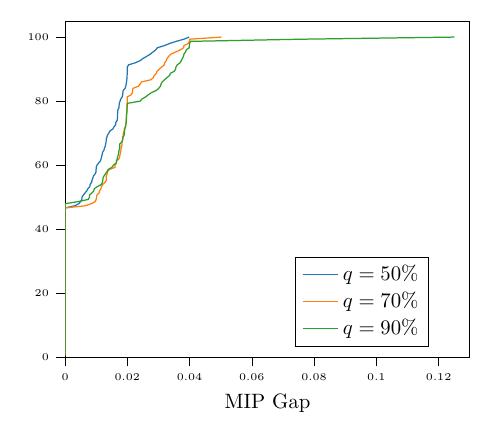}
\caption{ MIP Gaps.}
\end{subfigure}
     \caption{Performance profile of our computational experience for the different values of parameter $q$. \label{fig:Pf}}
\end{figure}

Figure \ref{fig:Pf} depicts performance profiles with respect to CPU times and MIP Gaps. In these pictures, we plots the percentage of instances solved up to each value of the CPU or MIP Gap, respectively. We analyze the computational performance of the model for the different values of the parameter $q$. As one can observe, the model seems to perform slightly better for the the smaller values of parameter $q$, although the differences are tiny. For this reason, in what follows we will not differentiate between the different values of the parameter $q$.

In Table \ref{t:compres} we summarize the results of our computational experiments. The first three columns indicate the size of the instances. Each of the rows summarize the results of a total of 15 instances (five random instances and three values of the parameter $q$). Column \textbf{CPUTime (secs)} is the average CPU time, in seconds, that the solver required to solve the instances. Column \textbf{MIPGap} indicates the average percent MIP Gap. In case the instance is optimally solved, the MIP Gap is $0\%$. Otherwise, this number reports either the MIPGap obtained within the time limit (if it is greater than the MIP Gap limit) or the MIP Gap when the MIP Gap limit was reached (that can be slightly smaller than the MIP Gap limit). Finally, column \textbf{UnSolved} indicates the percent instances summarized in the row that were not optimally solved or reaching the MIP Gap limit within the time limit.

\begin{table}[H]
  \centering
    \adjustbox{width=0.8\textwidth}{\begin{tabular}{ccc|rrr}
    $n$ & $m$& $d$ &\textbf{CPUTime (secs)} & \textbf{MIPGap} & \textbf{\%UnSolved}\\
\hline
\textbf{25} & \textbf{5} & 10    & 0.67  & 0\% & 0\% \\
      & \textbf{10} & 10    & 1.71  & 0\% & 0\% \\
\hline
\textbf{25 Total} &       &       & \textbf{1.19} & \textbf{0\%} & \textbf{0\%} \\
\hline
\textbf{50} & \textbf{5} & 10    & 1.58  & 0\% & 0\% \\
      &       & 20    & 2.24  & 0\% & 0\% \\
      & \textbf{10} & 10    & 4.55  & 0\% & 0\% \\
      &       & 20    & 4.70  & 0\% & 0\% \\
      & \textbf{20} & 10    & 24.78 & 0\% & 0\% \\
      &       & 20    & 28.83 & 0\% & 0\%\\
\hline
\textbf{50 Total} &       &       & \textbf{11.11} & \textbf{0\%} & \textbf{0\%} \\
\hline
\textbf{100} & \textbf{10} & 10    & 32.08 & 0\% & 0\% \\
      &       & 20    & 48.22 & 0\% & 0\% \\
      &       & 50    & 28.14 & 0\% & 0\% \\
      & \textbf{20} & 10    & 109.59 & 0\% & 0\% \\
      &       & 20    & 181.59 & 0\% & 0\% \\
      &       & 50    & 179.04 & 0\% & 0\% \\
\hline
\textbf{100 Total} &       &       & \textbf{96.44} & \textbf{0\%} & \textbf{0\%} \\
\hline
\textbf{200} & \textbf{20} & 10    & 444.32 & 1.47\% & 0\% \\
      &       & 20    & 137.37 & 1.69\% & 0\% \\
      &       & 50    & 70.65 & 1.44\% & 0\% \\
      &       & 100   & 51.63 & 1.40\% & 0\% \\
      & \textbf{50} & 10    & 1195.32 & 1.77\% & 0\% \\
      &       & 20    & 413.42 & 1.81\% & 0\% \\
      &       & 50    & 46.63 & 1.54\% & 0\% \\
      &       & 100   & 40.78 & 1.45\% & 0\% \\
\hline
\textbf{200 Total} &       &       & \textbf{300.01} & \textbf{1.57\%} & \textbf{0\%} \\
\hline
\textbf{500} & \textbf{50} & 10    & 6946.33 & 4.45\% & 20.00\%\\
      &       & 20    & 1041.76 & 2.90\% & 0\% \\
      &       & 50    & 912.06 & 1.93\% & 0\% \\
      &       & 100   & 713.42 & 1.74\% & 0\% \\
      & \textbf{100} & 10    & 5656.14 & 3.10\% & 6.67\% \\
      &       & 20    & 3625.32 & 2.50\% & 0\% \\
      &       & 50    & 2366.02 & 2.47\% & 0\% \\
      &       & 100   & 3195.16 & 1.32\% & 0\% \\
\hline
\textbf{500 Total} &       &       & \textbf{3057.03} & \textbf{2.55\%} & \textbf{3.33\%} \\
\hline
\textbf{Total} &       &       & \textbf{916.80} & \textbf{1.10\%} & \textbf{0.89\%} \\
\end{tabular}}%
  \caption{Summary of our computational experience.\label{t:compres}}
\end{table}%

As mentioned before, we observe that all the instances  with up to $n=100$ have been optimally solved within the time limit. As expected, the computing time increases with the value of $n$ and $m$, while it does not seem to depend on the value of $d$. The average computing time over all these instances is approximately  $46.26$ seconds, being $1.19$ seconds for the instances with $n=25$ , $11.11$ seconds for $n=50$, and $96.44$ for $n=100$. The average  computing time needed to obtain an optimal solution depending on the value of $m$ is $1.5$ seconds for $m=5$, $19.9$ seconds for $m=10$, and $104.8$ seconds for $m=20$.  

Regarding the instances with $n=200$ we observe that all the instances have been solved, which in this case means that all the instances have reached the MIP gap limit within the time limit, being the average consuming time over all the  instances $300$ seconds and the average MIP gap $1.57\%$.  One can observe that, for these instances, the computing time is greater as the lower is the value of the parameter $d$. This may be because  the value of parameter $d$ does not affect to the number of binary variables in our improved formulation and it might happen that the smaller the number of EC and injection points, the more difficult it is to decide where to install the different plants to satisfy them at optimal cost.  This fact can also be observed for the largest instances with $n=500$. For $d=10$, $20\%$ of the instances with $m=50$ and $6.67\%$ of the instances with $m=100$  have not been solved within the time limit, that is, these instances have not reached the MIP gap of $4\%$ in 6 hours. The rest of the instances have been solved within the time limit. The average computing time for all the instances with $n=500$ was $3057$ seconds and the average MIP gap, $2.55\%$.

The results shown in  Table \ref{t:compres}, together with Figure \ref{fig:Pf}, allow us to get a clear empirical evidence that using our improved formulation, the problem  can be solved (with the MIP gap limit) within the time limit using a commercial solver.   In fact, only 4 of the 450 instances (all of them with $n=500$) have reached the time limit  with a MIP Gap greater than the MIP  Gap limit. This number represents less than $1 \%$ of the total instances and only $3.3\%$ of the instances with $n=500$.

From our experiments, we conclude that our approach provides a useful tool to make decisions about on the complex logistic problem behind the energy transformation of waste into biogas in realistic instances using reasonable computational resources. Thus, our model can also be  used to evaluate different alternatives based on different values for the set-up costs for the plants or different transportation costs.

\section{Case Study \label{sec:case-study}}

\begin{figure}[h!]
\centering\includegraphics[scale=0.8]{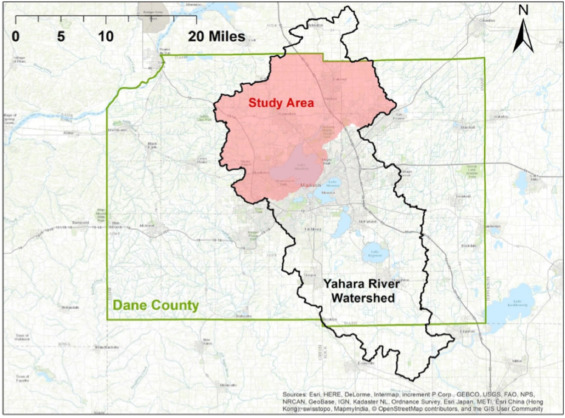}
\caption{Lake Mendota in the Upper Yahara watershed region in Dane County, WI (source: \citep{MA2022107832}).\label{fig:yahara}}
\end{figure}

We tested our model in a real-world dataset based on the region of upper Yahara watershed region in the State of Wisconsin (see Figure \ref{fig:yahara}) whose data is available at  \url{https://github.com/zavalab/JuliaBox/tree/master/Graph\_S\%26C}. A detailed description of this dataset can be found in \citep{MA2022107832,SAMPAT2019352,SAMPAT2021105199,TOMINAC2022107666}. This region consists of  203 dairy farms whose waste production is known and whose locations are predefined. These locations determine our set of WS and the potential positions for the PT plants.   The potential positions for the AD plants and the BL were randomly generated in that region. The positions of the injections points on the GPN were also randomly generated in this region  and a network connecting this points was constructed.  The position of the EC where randomly generated  outside that region but in the minimum rectangular box containing it. 

In Figure \ref{fig:input} we plot the coordinates of the different agents involved in the W2BLP. Red points in the plot indicate the WS locations (as larger the size of the dot, larger the production of waste in the farm)  and the potential locations for PT, blue points are the potential locations for AD plants and for BL plants, green points indicate the EC locations, grey dotted lines represent the GPN and grey points the injection points of the GPN.

\begin{figure}
\centering\includegraphics[scale=0.5]{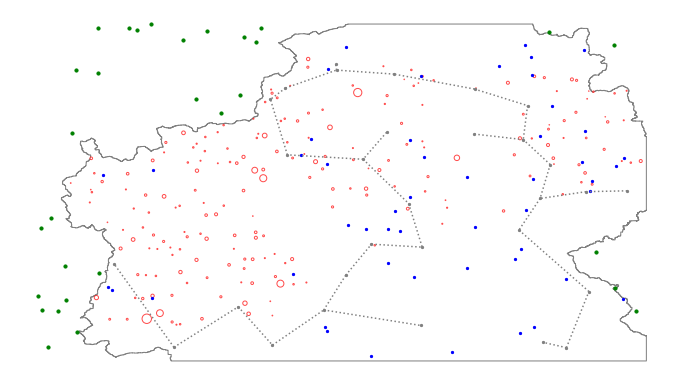}
\caption{Input Data of the Case Study. \label{fig:input}}
\end{figure}

The parameters not provided in the above repository were fixed as follows:

\begin{itemize}
\item $c^1_{ij}=0.3\times {\rm dist}_{ij}$, 
for $i\in N$ and $j \in P_1$.
\item $c^2_{jk}=0.15 \times {\rm dist}_{jk}$, 
for $j\in P_1$ and $k \in P_2$.
\item  $c^N_{ji}= 0.15\times {\rm dist}_{ji}$, 
for $j\in P_2$ and $i \in N$.
\item $D_\ell$ was randomly generated in $[\delta \gamma^1 \beta\times \min w_i, \delta \gamma^1 \beta \times \max w_i]$.
\item The set-up budget was fix to $\kappa \times \widehat{B}$, for $\kappa \in \{0.1, 0.2\}$, where $\widehat{B}$ was computed as described in Section \ref{sec:comput}:
$$\widehat{B} = \begin{cases}
    472.33 & \mbox{if $q=50\%$,}\\
    472.45 & \mbox{if $q=70\%$,}\\
    472.32 & \mbox{if $q=90\%$,}\\
\end{cases}$$
\item The rest of parameters  have  been considered as described in  Section \ref{sec:comput}  (q ranges in
\{50\%, 70\%, 90\%\}, $p=1$, $\delta=0.8$, $\gamma_1=0.8$, $\gamma_2=0.15$, $\beta=0.7$, $\alpha_\ell= \frac{D_{\ell}}{\sum_{l\in C} D_l}$ for all $\ell\in C$, and $R_i=\frac{w_i}{W}$ for all $i\in N$). 
\end{itemize}
The parameters required to reproduce the obtained results are available at \url{https://github.com/vblancoOR/w2blp}.

\begin{table}
\centering

\begin{tabular}{| l | l | l | l |}
\hline
$B$ & $q$ & \textbf{CPU Time} & \textbf{MIPGap} \\
\hline
\multirow{3}{*}{$0.1 \widehat{B}$} & {0.5} & 2595.82 & 0\% \\
 & {0.7} & 18093.03 & 0\% \\
 & {0.9} & 21602.8 & 0.1\% \\\hline
\multirow{3}{*}{$0.2 \widehat{B}$} & {0.5} & 84.67 & 0\% \\
 & {0.7} & 473.81 & 0\% \\
 & {0.9} & 3106.91 & 0\% \\\hline
\end{tabular}
\caption{CPU times (in seconds) and MIP gaps for solving the instances in the case study.\label{tab_cs_times}}
\end{table}

As in the previous section, we set a time limit of 6 hours for solving the problem. 
In Table \ref{tab_cs_times} we show the CPU times (in seconds) and MIP Gaps obtained after running our model to the six different configurations of budget and $q$. In that table, one can observe that except the case $B=0.1\widehat{B}$ with $q=0.9$, all the instances were optimally solved within the time limit. The case were the optimality is not guaranteed obtained a MIP gap of $0.1\%$ which is insignificant. Furthermore, as expected, a more limited budget has a significative impact in the CPU time required to solve the problem, being the problems with budget $0.1\widehat{B}$ more challenging than those with budget $0.2\widehat{B}$.

In Figures \ref{fig:Location}, \ref{fig:Flow1andN}, and \ref{fig:FlowFromP_2} we show the results of our experiment for the Yahara watershed dataset for the different choices of budget, $B$, and $q$, and the different decisions that are made by our model.

In Figure \ref{fig:Location}, we show the results about the optimal locations for the PT plants (thick red points), for the AD plants (blue squares), for the BL plants (blue triangles) and for the links connecting AD plants with BL plants (blue lines) and AD plants with injection points of the GPN (black lines). As can be observed, the budget and the percent of production that must be injected in the GPN has a direct impact in the structure of the obtained solutions. On the one hand, as expected, as smaller the budget, smaller the number of plants and pipelines that are open. In particular, the budget mostly affect to the construction of pipelines to either inject to the GPN or transport to the BL, but also in the number of installed plants. We also observed that the network of the different installed plants and pipelines has more connected components for the larger budget, whereas is almost connected for the small budget. Regarding the value of $q$, it seems that as larger the value of $q$ closer the new facilities to the GPN.

\begin{figure}[h!]
     \begin{subfigure}[b]{0.45\textwidth}
         \centering
         \includegraphics[width=\textwidth]{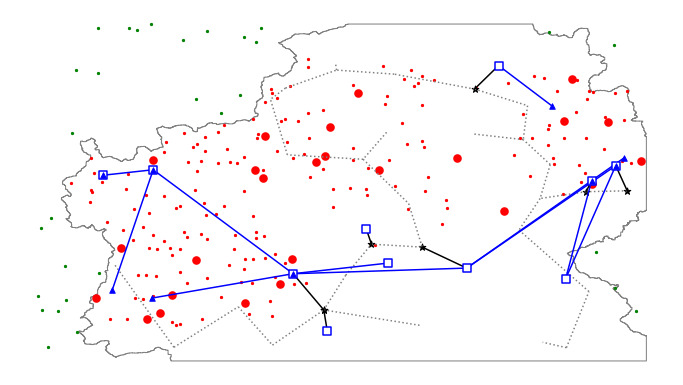}
         \caption{\scriptsize $q=50\%$\quad  and $\mathrm{B} = 0.2  \widehat{B}$. \label{fig:loc52}}
     \end{subfigure}\quad
     \begin{subfigure}[b]{0.45\textwidth}
         \includegraphics[width=\textwidth]{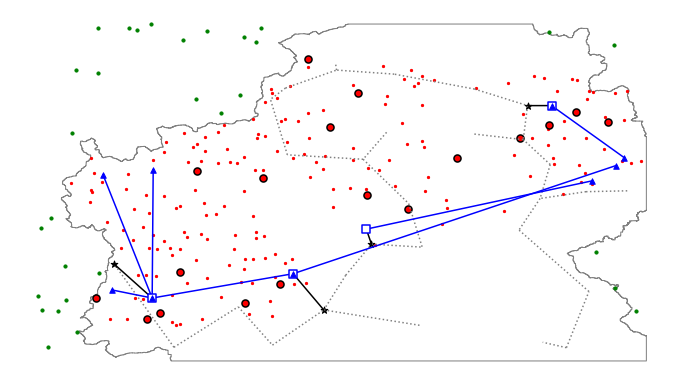}
         \caption{\scriptsize $q=50\%$\quad  and $\mathrm{B} = 0.1  \widehat{B}$. \label{fig:loc51}}
     \end{subfigure}\\
     \begin{subfigure}[b]{0.45\textwidth}
         \includegraphics[width=\textwidth]{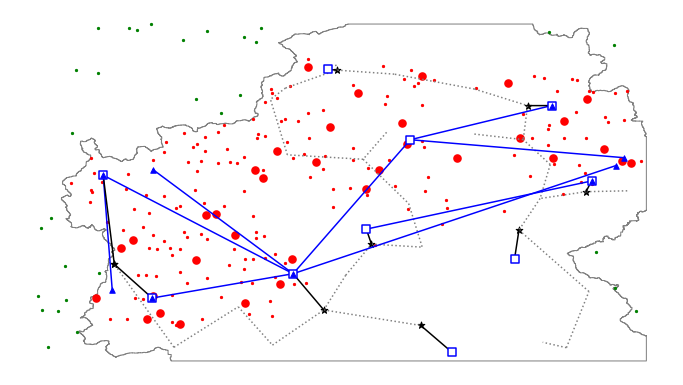}
         \caption{\scriptsize $q=70\%$\quad  and $\mathrm{B} = 0.2  \widehat{B}$. \label{fig:loc72}}
         \end{subfigure}\quad
     \begin{subfigure}[b]{0.45\textwidth}
         \includegraphics[width=\textwidth]{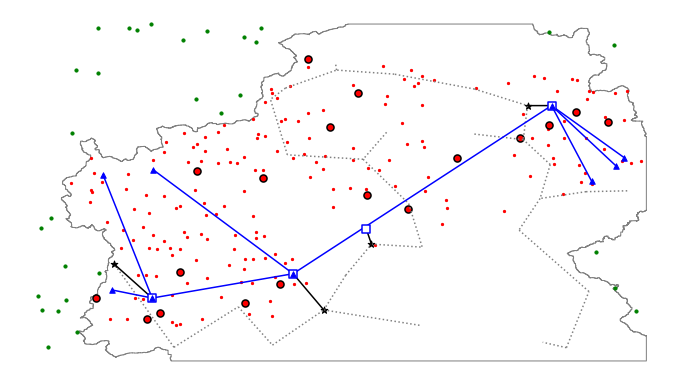}
         \caption{\scriptsize $q=70\%$\quad  and $\mathrm{B} = 0.1  \widehat{B}$. \label{fig:loc71}}
     \end{subfigure}\\
     \begin{subfigure}[b]{0.45\textwidth}
         \includegraphics[width=\textwidth]{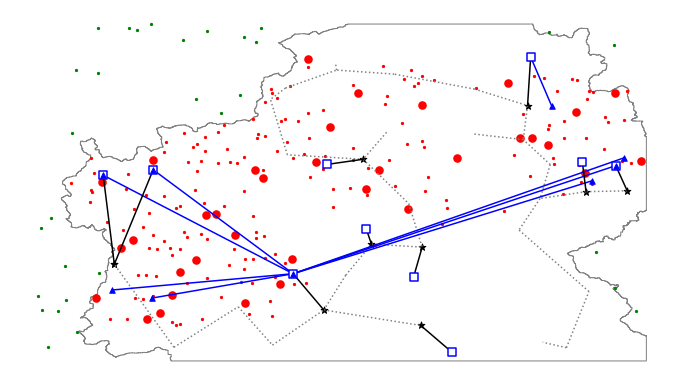}
         \caption{\scriptsize $q=90\%$\quad  and $\mathrm{B} = 0.2  \widehat{B}$. \label{fig:loc92}}
         \end{subfigure}\quad
     \begin{subfigure}[b]{0.45\textwidth}
         \includegraphics[width=\textwidth]{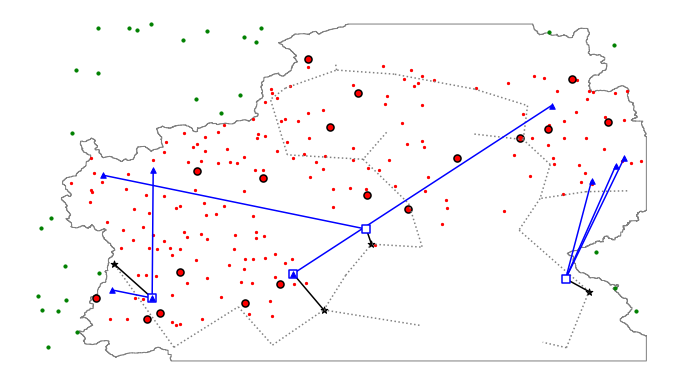}
         \caption{\scriptsize $q=90\%$\quad  and $\mathrm{B} = 0.1  \widehat{B}$. \label{fig:loc91}}
     \end{subfigure}
     \caption{Optimal location of the different types of plants and pipelines for different values of $B$ and $q$. \label{fig:Location}}
    \end{figure}

In Figure \ref{fig:Flow1andN} we show the distribution network  between the WS plants, the PT plants, and the AD plants. Specifically, the links for which there is a positive waste flow from WS to PT plants (red lines), a DOW flow from PT plants to AD plants (purple lines) and digester solid flow from AD plants to WS (yellow lines). The main observation that can be drawn from these plots is that in most of the cases there are differentiated clusters of WS plants (farms) that share the same PT plants and AD plants. This observations implies that for larger datasets where the model could not be able to solve the problem, the optimal solution would be adequately approximated by clustering the WS (with an adequate criterion) and solve the problem separately for each cluster. 

\begin{figure}
     \begin{subfigure}[b]{0.45\textwidth}
         \centering
         \includegraphics[width=\textwidth]{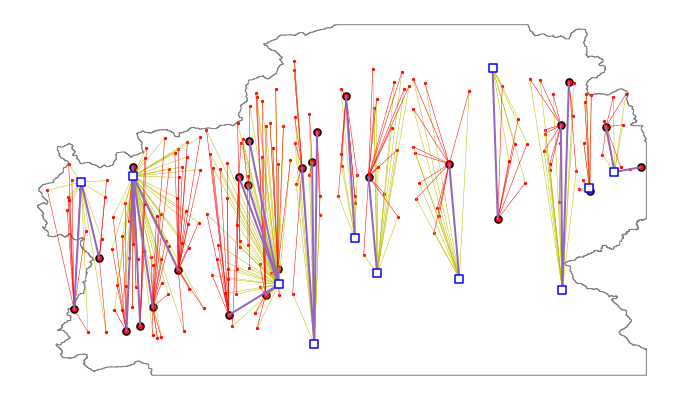}
         \caption{\scriptsize $q=50\%$\quad  and $\mathrm{B} = 0.2  \widehat{B}$. \label{fig:Flow1andN52}}
     \end{subfigure}\quad
     \begin{subfigure}[b]{0.45\textwidth}
         \includegraphics[width=\textwidth]{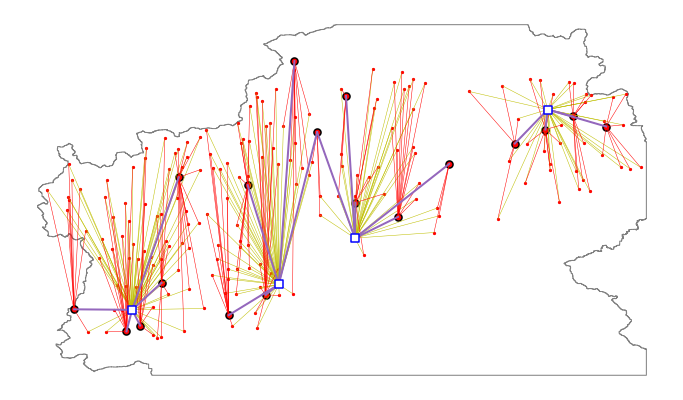}
         \caption{\scriptsize $q=50\%$\quad  and $\mathrm{B} = 0.1  \widehat{B}$. \label{fig:Flow1andN51}}
     \end{subfigure}\\
     \begin{subfigure}[b]{0.45\textwidth}
         \includegraphics[width=\textwidth]{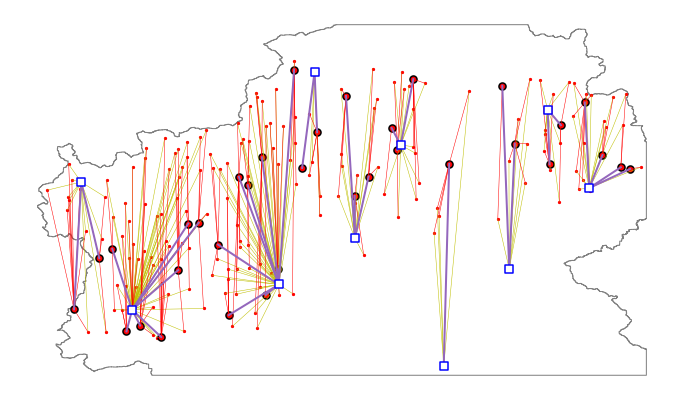}
         \caption{\scriptsize $q=70\%$\quad  and $\mathrm{B} = 0.2  \widehat{B}$. \label{fig:Flow1andN72}}
         \end{subfigure}\quad
     \begin{subfigure}[b]{0.45\textwidth}
         \includegraphics[width=\textwidth]{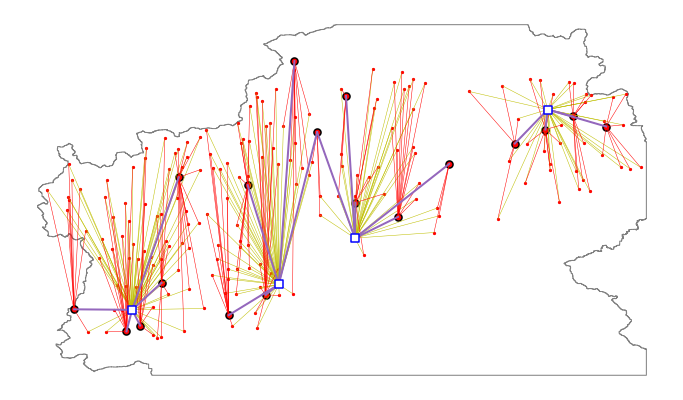}
         \caption{\scriptsize $q=70\%$\quad  and $\mathrm{B} = 0.1  \widehat{B}$. \label{fig:Flow1andN71}}
     \end{subfigure}\\
     \begin{subfigure}[b]{0.45\textwidth}
         \includegraphics[width=\textwidth]{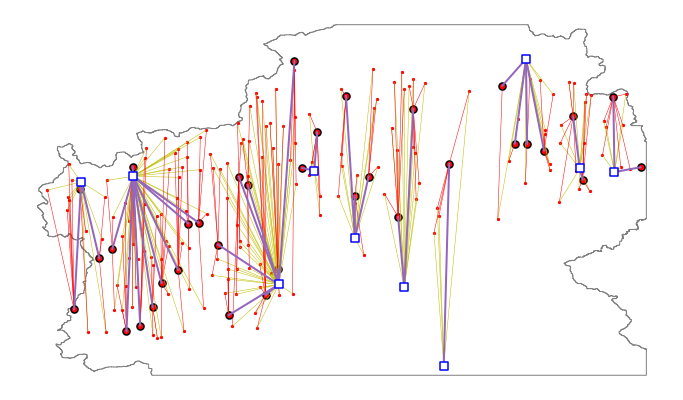}
         \caption{\scriptsize $q=90\%$\quad  and $\mathrm{B} = 0.2  \widehat{B}$. \label{fig:Flow1andN92}}
         \end{subfigure}\quad
     \begin{subfigure}[b]{0.45\textwidth}
         \includegraphics[width=\textwidth]{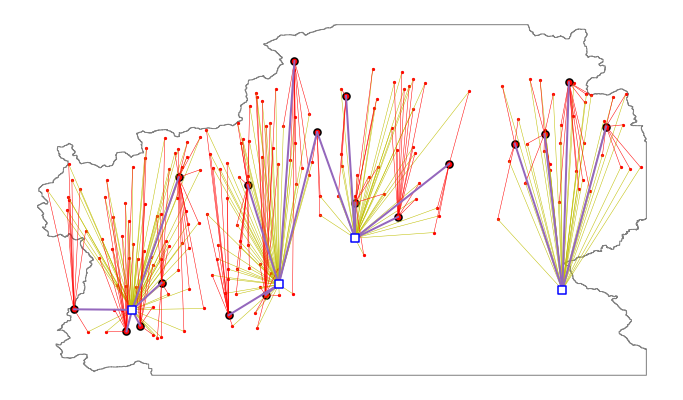}
         \caption{\scriptsize $q=90\%$\quad  and $\mathrm{B} = 0.1  \widehat{B}$. \label{fig:Flow1andN91}}
     \end{subfigure}
     \caption{Distribution network between WS, PT plants, and AD plants.\label{fig:Flow1andN}}
    \end{figure}

Finally, in Figure \ref{fig:FlowFromP_2}, we show the biomethane flow from AD plants to BL plants (blue lines) and to injection points of the GPN (black lines), and LGN flow from BL plants to EC (green lines). Note that some of the AD plants are devoted only to give service to the GPN, others that serve exclusively the BL plants (and then the EC), but one can also find AD plants that send part of the production to the GPN and what remains, to the BL plants. This behaviour would have never happen if an integrated model as the one we propose would not have been considered. Note also that a single AD plant is allowed to send biomethane to different BL plants.

\begin{figure}
     \begin{subfigure}[b]{0.45\textwidth}
         \centering
         \includegraphics[width=\textwidth]{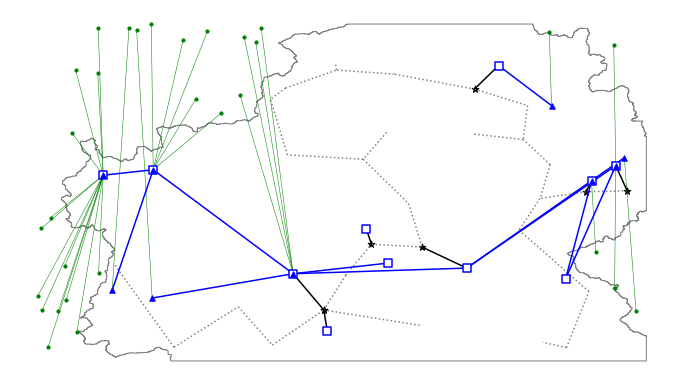}
         \caption{\scriptsize $q=50\%$\quad  and $\mathrm{B} = 0.2  \widehat{B}$. \label{fig:FlowFromP_252}}
     \end{subfigure}\quad
     \begin{subfigure}[b]{0.45\textwidth}
         \includegraphics[width=\textwidth]{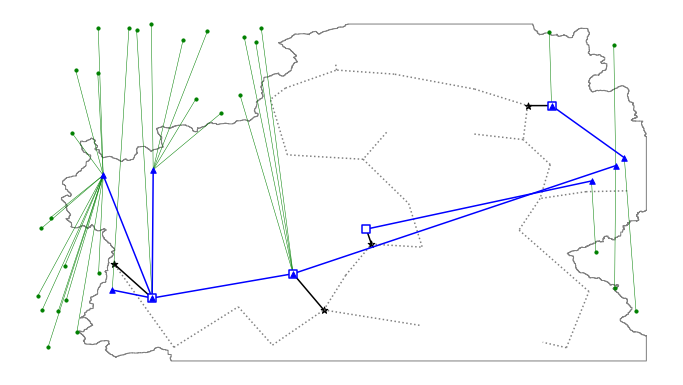}
         \caption{\scriptsize $q=50\%$\quad  and $\mathrm{B} = 0.1  \widehat{B}$. \label{fig:FlowFromP_251}}
     \end{subfigure}\\
     \begin{subfigure}[b]{0.45\textwidth}
         \includegraphics[width=\textwidth]{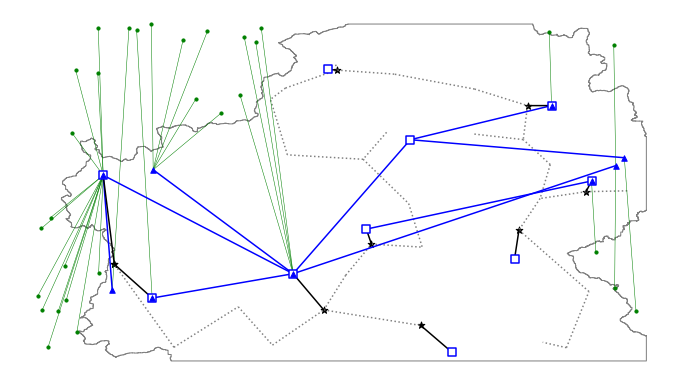}
         \caption{\scriptsize $q=70\%$\quad  and $\mathrm{B} = 0.2  \widehat{B}$. \label{fig:FlowFromP_272}}
         \end{subfigure}\quad
     \begin{subfigure}[b]{0.45\textwidth}
         \includegraphics[width=\textwidth]{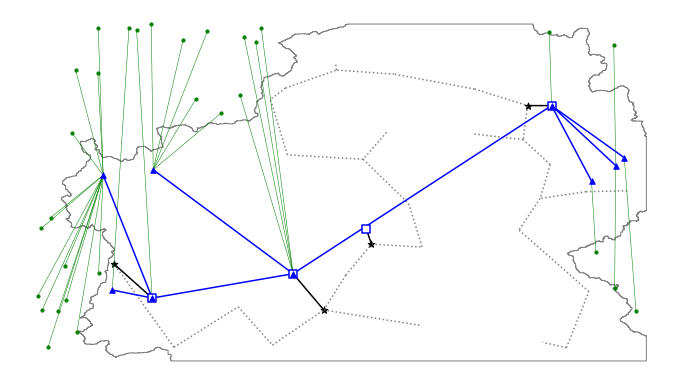}
         \caption{\scriptsize $q=70\%$\quad  and $\mathrm{B} = 0.1  \widehat{B}$. \label{fig:FlowFromP_271}}
     \end{subfigure}\\
     \begin{subfigure}[b]{0.45\textwidth}
         \includegraphics[width=\textwidth]{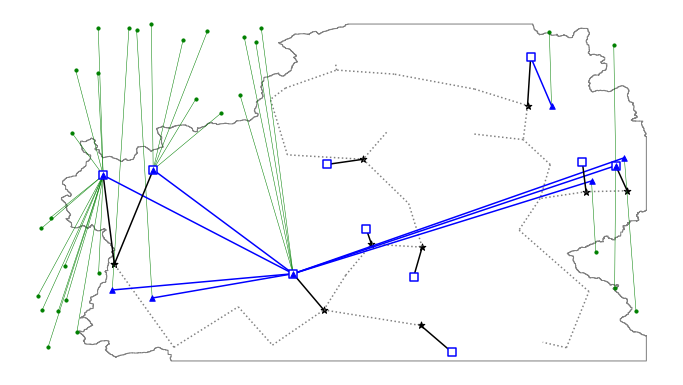}
         \caption{\scriptsize $q=90\%$\quad  and $\mathrm{B} = 0.2  \widehat{B}$. \label{fig:FlowFromP_292}}
         \end{subfigure}\quad
     \begin{subfigure}[b]{0.45\textwidth}
         \includegraphics[width=\textwidth]{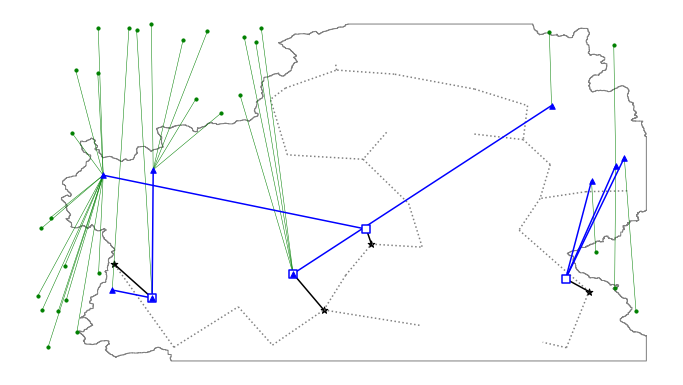}
         \caption{\scriptsize $q=90\%$\quad  and $\mathrm{B} = 0.1  \widehat{B}$. \label{fig:FlowFromP_291}}
     \end{subfigure}
     \caption{Distribution network between AD plants, BL plants, injection points of the GPN, and EC.\label{fig:FlowFromP_2}}
    \end{figure}

\section{Conclusions \label{sec:conclu}}

In this paper, we propose a mathematical optimization approach to design an optimal logistic system to distribute the different products involved in the generation of biogas from waste. Specifically, given a set of waste storage centers, an existing gas pipeline network, and a set of external customers, we provide a decision aid tool to determine the number and optimal locations of pre-treatment plants, anaerobic digestion plants, and biomethane liquefaction plants, as well as the pipelines linking some of these plants. Additionally, we provide a distribution plan to send the different wastes, to serve either the gas pipeline network or the external customers with the produced biogas, and to serve the waste storage centers with the generated fertilizer. All the decisions are made by minimizing the transportation costs of the different products and restricting the installation of plants and pipelines to a given budget.

We develop a new mixed integer linear programming formulation for the problem and prove some results that allow us to reduce the size of the model. With this simplification, we were able to obtain optimal solutions for the problem in real-world intances.

We report the results of an extensive battery of synthetic computational experiments, concluding that our approach is suitable to be applied to different settings and sizes. We also analyze the case study of the Yahara watershed, and study the obtained solutions based on different parameters.

Our future research on this topic includes the incorporation of uncertainty in some parameters of the model. Concretely, the production of waste in the farms or waste storages is known to vary in the different seasons of the year, and this production may has an impact in the solution of the problem. We will design stochastic optimization models that take into account this uncertainty to construct robust solution of the W2BLP. Additionally, we will consider capacity constraints for the different plants and multiple periods for the decisions made in our model. Instead of assuming that the plants and pipelines are installed here-and-now, we will decide in which period of a time horizon each of the installations is constructed and used, assuming that they have a limited capacity and that a budget is available for each period. The integrated model that consider uncertainty, capacity constraints and multiperiod decisions will be closer to reality, but the mathematical programming model for it would be prohibitive even for small instances. Thus, a different solution approach will have to be designed to solve it, that would not be exact, but heuristic.

\section*{Acknowledgement}
This research has been partially supported by Spanish Ministerio de Ciencia e Innovación, AEI/FEDER grant number PID 2020 - 114594GBC21, RED2022-134149-T (Thematic Network on Location Science and Related Problems); FEDER+Junta de Andalucía projects B-FQM-322-UGR20, C‐EXP‐139‐UGR23, and AT 21\_00032; VII PPIT-US (Ayudas Estancias Breves, Modalidad III.2A); and the IMAG-Maria de Maeztu grant CEX2020-001105-M /AEI /10.13039/501100011033. The authors also acknowledge the support of the Wisconsin Institute of Discovery (University of Wisconsin-Madison) for using its computational resources.

\end{document}